\documentclass[12pt]{amsart}
\usepackage{amsmath}
\usepackage{mathtools}
\usepackage{amssymb}
\usepackage{amsthm}
\usepackage{enumerate}
\usepackage{graphicx,color}
\usepackage{xspace}
\usepackage{verbatim,url}
\RequirePackage[colorlinks,citecolor=blue,urlcolor=blue]{hyperref}

\newtheorem{theorem}{Theorem}
\newtheorem{lemma}[theorem]{Lemma}
\newtheorem{corollary}[theorem]{Corollary}

\newtheorem{proposition}[theorem]{Proposition}

\newcommand{\ind}{\mathbf{1}}

\newcommand{\E}{\mathbb{E}}
\renewcommand{\P}{\mathbb{P}}

\newcounter{notecounter}

\newcommand{\IGNORE}[1]{}

\newcounter{mycount}
\newenvironment{ilist}{\begin{list}{\rm(\roman{mycount})}%
   {\usecounter{mycount}\labelwidth=1cm\itemsep 0pt}}{\end{list}}

\begin{document}

\title{Wald for non-stopping times: The rewards of impatient prophets}
\date{12 May 2014}

\author[Holroyd]{Alexander E.\ Holroyd}
\address{(AEH) Microsoft Research, 1 Microsoft Way, Redmond, WA 98052, USA}
\email{\href{mailto:holroyd at microsoft.com}{holroyd at microsoft.com}}
\urladdr{\url{http://research.microsoft.com/~holroyd/}}

\author[Peres]{Yuval Peres}
\address{(YP) Microsoft Research, 1 Microsoft Way, Redmond, WA 98052, USA}
\email{\href{mailto:peres@microsoft.com}{peres@microsoft.com}}
\urladdr{\url{http://research.microsoft.com/~peres/}}

\author[Steif]{Jeffrey E.\ Steif}
\address{(JES) Chalmers University of Technology and G\"{o}teborg University,
SE-41296 Gothenburg, Sweden}
\email{\href{mailto:steif@math.chalmers.se}{steif@math.chalmers.se}}
\urladdr{\url{http://www.math.chalmers.se/~steif/}}

%\title{When does Wald's Theorem hold for {\em non-}stopping times?
%}

\begin{abstract}
Let $X_1,X_2,\ldots$ be independent identically distributed nonnegative
random variables.  Wald's identity states that the random sum
$S_T:=X_1+\cdots+X_T$ has expectation $\E T \cdot \E X_1$ provided $T$ is a
stopping time. We prove here that for any $1<\alpha\leq 2$,
if $T$ is an arbitrary nonnegative random variable, then
$S_T$ has finite expectation provided that
$X_1$ has finite $\alpha$-moment and $T$ has finite $1/(\alpha-1)$-moment.
We also prove a variant in which $T$ is assumed to have a finite exponential moment.
These moment conditions are sharp in the sense that for any i.i.d.\ sequence
$X_i$ violating them, there is a $T$ satisfying the given condition for which
$S_T$ (and, in fact, $X_T$) has infinite expectation.

An interpretation of this is given in terms of a prophet being more rewarded than a gambler
when a certain {\em impatience} restriction is imposed.
\end{abstract}

\maketitle

\section{Introduction}

Let $X_1,X_2,\ldots$ be independent identically distributed (i.i.d.)
nonnegative random variables, and let $T$ be a nonnegative integer-valued
random variable.  Write $S_n=\sum_{i=1}^n X_i$ and $X=X_1$.  Wald's identity
\cite{wald} states that if $T$ is a {\em stopping time} (which is to say that
for each $n$, the event $\{T=n\}$ lies in the $\sigma$-field generated by
$X_1,\ldots, X_n$), then
\begin{equation}\label{wald}
\E S_T = \E T \cdot  \E X.
\end{equation}
In particular, if $X$ and $T$ have finite mean then so does $S_T$.

It is natural to ask whether similar conclusions can be obtained if we drop
the requirement that $T$ be a stopping time.  It is too much to hope that the
equality \eqref{wald} still holds.  (For example, suppose that $X_i$ takes
values $0,1$ with equal probabilities, and let $T$ be $1$ if $X_2=0$ and
otherwise $2$.  Then $\E S_T=1\neq \tfrac32 \cdot \tfrac12 = \E T \cdot \E
X$.)  However, one may still ask when $S_T$ has finite mean.  It turns out
that finite means of $X$ and $T$ no longer suffice, but stronger moment
conditions do.  Our main result gives sharp moment conditions for this
conclusion to hold. In addition, when the moment conditions fail, with a
suitably chosen $T$ we can arrange that even the final summand $X_T$ has
infinite mean.  Here is the precise statement.  (If $T=0$ we take by
convention $X_T=0$).

\begin{theorem}\label{thm:main}
Let $X_1,X_2,\ldots$ be i.i.d.\ nonnegative random variables, and write
$S_n:=\sum_{i=1}^n X_i$ and $X=X_1$. For each $\alpha\in (1,2]$, the
following are equivalent.
\begin{ilist}
  \item $\E{X^\alpha}<\infty$.
  \item For every nonnegative integer-valued random variable $T$ satisfying
$\E T^{1/(\alpha-1)}<\infty$, we have $\E{S_T}<\infty$.
  \item For every nonnegative integer-valued random variable $T$ satisfying
$\E T^{1/(\alpha-1)}<\infty$, we have $\E{X_T}<\infty$.
\end{ilist}
\end{theorem}

The special case $\alpha=2$ of Theorem~\ref{thm:main} is particularly
natural: then the condition on $X$ in (i) is that it have finite variance,
and the condition on $T$ in (ii) and (iii) is that it have finite mean. At
the other extreme, as $\alpha\downarrow 1$, (ii) and (iii) require
successively higher moments of $T$ to be finite.  One may ask what happens
when $T$ satisfies an even stronger condition such as a finite exponential
moment -- what condition must we impose on $X$, if we are to conclude $\E
S_T<\infty$? The following provides an answer, in which, moreover, the
independence assumption may be relaxed.
\begin{theorem}\label{thm:exp}
Let $X_1,X_2,\ldots$ be i.i.d.\ nonnegative random variables, and write
$S_n:=\sum_{i=1}^n X_i$ and $X=X_1$. The following are equivalent.
\begin{ilist}
  \item $\E[X(\log X)^+]<\infty$.
  \item For every nonnegative integer-valued random variable $T$ satisfying
$\E e^{cT}<\infty$ for some $c>0$, we have $\E{S_T}<\infty$.
  \item For every nonnegative integer-valued random variable $T$ satisfying
$\E e^{cT}<\infty$ for some $c>0$, we have $\E{X_T}<\infty$.
\end{ilist}
Moreover, if $X_1,X_2,\ldots$ are assumed identically distributed but not
necessarily independent, then (i) and (ii) are equivalent.
\end{theorem}

On the other hand, in the following variant of Theorem~\ref{thm:main},
dropping independence results in a different moment condition for $T$.
\begin{proposition}\label{thm:GeneralMomentNOTiid}
Let $X$ be a nonnegative random variable. For each $\alpha\in (1,2]$, the
following are equivalent.
\begin{ilist}
\item $\E{X^\alpha}<\infty$.
\item\sloppypar For every nonnegative integer valued random variable $T$
    satisfying
    $\E T^{\alpha/(\alpha-1)}<\infty$, and for any $X_1,X_2,\ldots$
    identically distributed with $X$ (but not necessarily independent),
    we have $\E S_T<\infty$.
\end{ilist}
\end{proposition}

In order to prove the implications (iii) $\Rightarrow$ (i) of Theorems
\ref{thm:main} and \ref{thm:exp}, we will assume that (i) fails, and
construct a suitable $T$ for which $\E X_T=\infty$ (and thus also $\E
S_T=\infty$).
This $T$ will be the {\em last} time the random sequence
is in a certain (time-dependent) deterministic set, i.e.\
$$T:=\max\{n: X_n\in B_n\}$$
for a suitable sequence of sets $B_n$.  It is interesting to note that, in contrast, no
$T$ of the form $\min\{n: X_n\in B_n\}$ could work for this purpose, since
 such a $T$ is a stopping time, so Wald's identity applies. In the context of
Theorem \ref{thm:exp}, $T$ will take the form
$$T:=\max\{n: X_n\geq f(n)\}$$ for a suitable function $f$.

The results here bear an interesting relation to so-called prophet
inequalities; see \cite{HK} for a survey. A central prophet inequality (see \cite{KS}) states that if
$X_1,X_2,\ldots$ are independent (not necessarily identically distributed)
nonnegative random variables then
\begin{equation}\label{prophet}
\sup_{U\in \mathcal{U}} \E X_U \le 2\sup_{S\in \mathcal{S}} \E X_S,
\end{equation}
where $\mathcal{U}$ denotes the set of all positive integer-valued random
variables and $\mathcal{S}$ denotes the set of all stopping times.  The left
side is of course equal to $\E\sup_i X_i$.  The factor $2$ is sharp. The
interpretation is that a prophet and a gambler are presented
sequentially with the values $X_1,X_2,\ldots$, and each can stop at any time
$k$ and then receive payment $X_k$.  The prophet sees the entire sequence in
advance and so can obtain the left side of \eqref{prophet} in expectation,
while the gambler can only achieve the supremum on the right.  Thus
\eqref{prophet} states that the prophet's advantage is at most a factor of
$2$.

The inequality \eqref{prophet} is uninteresting when $(X_i)$ is an infinite
i.i.d.\ sequence, but for example applying it to $X_i \ind[i\leq n]$ (where
$n$ is fixed and $(X_i)$ are i.i.d.) gives
\begin{equation}\label{prophetbdd}
\sup_{\substack{U\in \mathcal{U}:\\U\leq n}}
\E X_U \le 2\sup_{\substack{S\in \mathcal{S}:\\S\leq n}} \E X_S,
\end{equation}
(and the factor of $2$ is again sharp). How does this result change if we
replace the condition that $U$ and $S$ are bounded by $n$ with a moment
restriction?  It turns out that the prophet's advantage can become infinite,
in the following sense.  Let $X_1,X_2,\ldots$ be any i.i.d.\ nonnegative
random variables with mean $1$ and infinite variance.  By
Theorem~\ref{thm:main}, there exists an integer-valued random variable $T$ so
that $\mu:=\E{T}< \infty$ but $\E{X_T}= \infty$.  Then we have
$$\sup_{\substack{U\in \mathcal{U}:\\ \E U\leq \mu}} \E X_U =\infty;
 \qquad
 \sup_{\substack{S\in \mathcal{S}:\\ \E S\leq \mu}} \E X_S \le \mu.
$$
Here the first claim follows by taking $U=T$ and the second claim follows from Wald's
identity.

Interpreting impatience as meaning that the time we stop at has a mean of at most $\mu$, we see
that this impatience hurts the gambler much more than the prophet.

%\note{Perhaps add quantitative example of the last pair of statements - how does the first sup
%depend on $n$?}

Our proof of the implication (i) $\Rightarrow$ (ii) in Theorem \ref{thm:main}
will rely on a concentration inequality which is due to Hsu and Robbins
\cite{HR} for the important special case $\alpha=2$, and a generalization due
to Katz \cite{K} for $\alpha<2$.  For expository reasons, we include a
proof of the Hsu-Robbins inequality, which is simpler than the original
proof, and is an adaptation of that given in \cite{CT}.  Thus, we give a
complete proof from first principles of Theorem \ref{thm:main} in the case
$\alpha=2$.  Erd\H{o}s \cite{E} proved a converse of the Hsu-Robbins result;
we will also obtain this converse in the case of nonnegative random variables
as a corollary of our results.

Throughout the article we will write $X=X_1$ and $S_n:=\sum_{i=1}^n X_i$.  If
$T=0$ then we take $X_T=0$ and $S_T=0$.

\section{The case of exponential tails}

In this section we give the proof of Theorem~\ref{thm:exp}, which is
relatively straightforward.  We start with a simple lemma relating $X_T$ and
$S_T$ for $T$ of the form that we will use for our counterexamples.  The same
lemma will be used in the proof of Theorem~\ref{thm:main}.

\begin{lemma}\label{last}
Let $X_1,X_2,\ldots$ be i.i.d.\ nonnegative random variables. Let $T$ be
defined by
$$T=\max\{k: X_k\in B_k\}$$
for some sequence of sets $B_k$ for which this set is a.s.\ finite, and where we take
$T=0$ and $X_T=0$ when the set is empty.  Then
$$\E S_T= \E [(T-1)^+]\cdot\E X + \E X_T.$$
\end{lemma}

\begin{proof}
Observe that $\ind[T=k]$ and $S_{k-1}$ are independent for every $k\geq 1$.
Therefore,
\begin{align*}
\E S_T &= \E\sum_{k=1}^\infty S_k \ind[T=k]\\
&=\E\sum_{k=1}^\infty (S_{k-1}+X_k)\ind[T=k]\\
&=\sum_{k=1}^\infty \E S_{k-1}\cdot \P(T=k) + \E\sum_{k=1}^\infty X_k \ind[T=k]\\
&=\E [(T-1)^+]\cdot \E X + \E X_T.\qedhere
\end{align*}
\end{proof}

%\begin{proof}[Proof of Theorem~\ref{thm:ExpTails}, (ii)$\Rightarrow$(iii)]
%This is immediate since $S_T\geq X_T$.
%\end{proof}

\begin{proof}[Proof of Theorem~\ref{thm:exp}]
We first prove that (i) and (ii) are equivalent, assuming only that the $X_i$
are identically distributed (not necessarily independent).

Assume that (i) holds, i.e.\ $\E[X(\log X)^+]<\infty$, and that $T$ is a
nonnegative integer-valued random variable satisfying $\E e^{cT}<\infty$.
Observe that $X_k\leq e^{ck}+X_k\ind[X_k >e^{ck}]$, so
$$
S_T\le \sum_{k=1}^T  e^{ck}+  \sum_{k=1}^T X_k \ind[X_k\ge e^{ck}].
$$
The first sum equals
$$
\frac{e^c(e^{cT}-1)}{e^c-1}
$$
which has finite expectation.
The expectation of the second sum is at most
$$
\sum_{k=1}^\infty \E\bigl( X\ind [X\ge e^{ck}]\bigr)=
\E\sum_{k=1}^\infty X\ind[X\ge e^{ck}]
=\E \biggl(X \Bigl\lfloor \frac{(\log X)^+}{c}\Bigr\rfloor\biggr)
<\infty.
$$
Hence $\E{S_T}<\infty$ as required, giving (ii).

Now assume that (i) fails, i.e.\ $\E [X(\log X)^+]=\infty$, but (ii) holds
(still without assuming independence of the $X_i$).  Taking $T\equiv 1$ in
(ii) shows that $\E X<\infty$.  Now let
\begin{equation}\label{exp-T}
 T:=\max\{k\ge 1:X_k\ge e^k\},
\end{equation}
where $T$ is taken to be 0 if the set above is empty and $\infty$ if it is
unbounded.  Then
$$
\P(T\ge k)\le \sum_{i=k}^\infty \P(X_i\ge e^i)\le
%\sum_{i=k}^\infty \P(X\ge e^i)\le
\sum_{i=k}^\infty \frac{\E{X}}{e^i}
$$
by Markov's inequality. The last sum is $ (\E X) e^{1-k}/(e-1)$, and hence
$\E e^{ck}<\infty$ for suitable $c>0$ (and in particular $T$ is a.s. finite).
On the other hand,
$$
\E{S_T}=\E\sum_{k=1}^\infty  X_k \ind[k\le T]
\geq \E\sum_{k=1}^\infty  X \ind[X\ge e^k]
=\E\bigl(X\lfloor (\log X)^+\rfloor\bigr),
$$
which is infinite, contradicting (ii).

Now assume that the $X_i$ are i.i.d.  We have already established that (i)
and (ii) are equivalent, and (ii) immediately implies (iii) since $S_T\geq
X_T$.  It therefore suffices to show that (iii) implies (i).  Suppose (i) fails and (iii) holds.  Taking $T\equiv1$ in (iii) shows that $\E X<\infty$.  Now
take the same $T$ as in \eqref{exp-T}.  As argued above, $\E S_T=\infty$ and
$\E e^{cT}<\infty$ for some $c >0$ (so $\E T<\infty$).  Hence (iii) gives $\E X_T<\infty$.  But
this contradicts Lemma~\ref{last}.
\end{proof}

{\bf Remark} Conditions (i) and (iii) cannot be equivalent if the i.i.d.\ condition is dropped
since if $X_1=X_2=X_3=\ldots$, then $X_T=X_1$ for every $T$ and so (iii) just corresponds to
$X$ having a first moment.

\section{The case $\alpha=2$ and the Hsu-Robbins Theorem}

In this section we prove Theorem~\ref{thm:main} in the important special case
$\alpha=2$ (so $1/(\alpha-1)=1$).   We will use the following result of Hsu
and Robbins \cite{HR}. For expository purposes we include a proof of this
result, which is simpler than the original proof, and is based on an argument
from \cite{CT}.

\begin{theorem}[Hsu and Robbins]\label{thm:HRE}
Let $X_1,X_2,\ldots$ be i.i.d.\ random variables with finite mean $\mu$ and
finite variance. Then for all $\epsilon>0$,
$$
\sum_{n=1}^\infty  \P\bigl(|S_n-n\mu|\ge n \epsilon\bigr) <\infty.
$$
\end{theorem}

\begin{proof}
We may assume without loss of generality that $\mu=0$ and $\E{X^2}=1$. Let
$X_n^*:=\max\{X_1,\ldots,X_n\}$ and $S_n^*:=\max\{S_1,\ldots,S_n\}$. Observe
that for any $h>0$, the stopping time $\tau_h:= \min{\{k : S_k \ge h\}}$
satisfies
\begin{equation}\label{eq:FirstStep}
\begin{multlined}
\P(S_n > 3h,\ X_n^* \le h) \\
\begin{aligned}
&\le \P(\tau_h<n,\ S_{\tau_h}\leq 2h) \;\P(S_n>3h \mid  \tau_h<n,\ S_{\tau_h}\leq 2h)\\
& \le \P(\tau_h \le n)^2
\end{aligned}
\end{multlined}
\end{equation}
where the last step used the strong Markov property at time $\tau_h$. Now
Kolmogorov's maximum inequality (see e.g.\ \cite[Lemma~4.15]{Kall}) implies
that
$$
\P(\tau_h \le n) =\P(S_n^* \ge h) \le \frac{\E{S^2_n}}{h^2}=\frac{n}{h^2}.
$$
Applying this with $h=\epsilon n/3$ we infer from (\ref{eq:FirstStep}) that
$$
\P\Bigl(S_n > n\epsilon,\ X_n^* \le \frac{\epsilon n}{3}\Bigr) \le
\frac{81}{\epsilon^4 n^2}.
$$
Moreover, we have
$$
\P\Bigl(X_n^*  >\frac{\epsilon n}{3}\Bigr) \le n\P\Bigl(X_1>\frac{\epsilon n}{3}\Bigr).
$$
Combining the last two inequalities give
$$
\P(S_n >n\epsilon) \le \frac{81}{\epsilon^4 n^2}
+ n \P\Bigl( X_1  >\frac{\epsilon n}{3}\Bigr).
$$
The first term on the right is summable in $n$, and the second term is
summable by the assumption of finite variance.  Applying the same argument to
$-S_n$ completes the proof.
\end{proof}

We will also a need a simple fact of real analysis, a converse to
H\"older's inequality, which we state in a probabilistic form.  See, e.g., Lemma 6.7 in \cite{Royden} for a related  statement.
The proof method is known as the ``gliding hump''; see \cite{Sokal} and the references therein.

%Again, for
%expository reasons we include a proof of the case that we need when
%$\alpha=2$ (which corresponds to the case $p=q=2$ in the lemma below).
%A proof of the general case may be easily constructed along similar
%lines.
% , or may be found for example in \cite{Royden}.
%
\begin{lemma}\label{cor:RealAnalysisGeneral}
Let $p,q>1$ satisfy $1/p+1/q=1$. Assume that a nonnegative random variable
$X$ satisfies $\E{Xg(X)}<\infty$ for every nonnegative function $g$ that
satisfies $\E{g^q(X)}<\infty$. Then $\E{X^p}<\infty$.
\end{lemma}

\begin{proof}%[Proof in the case $p=q=2$]
%It suffices to prove the result in the case when $X$ is integer valued.
%(Indeed, in the general case we may then consider the function $x\mapsto
%g(\lfloor x\rfloor)$ to conclude that $\E \lfloor X\rfloor^2<\infty$ and
%therefore $\E X^2<\infty$.)

Assume $\E{X^p}=\infty$. Letting $\psi_k:=\P(\lfloor X \rfloor =k)$, we have $\sum_{k=1}^\infty
\psi_k k^p =\E{\lfloor X \rfloor^p} =\infty$, so we can choose integers $0=a_0 <a_1 <a_2, \ldots$ such
that for each $\ell\ge 1$,
$$
S_\ell:=\sum_{k \in [a_{\ell-1},a_\ell)} \psi_k k^p\ge 1.
$$
Denote the interval $[a_{\ell-1}, a_{\ell})$ by $I_\ell$ and let $g$ be defined
on $[0,\infty)$ by
$$
g(x):=\frac{\lfloor x \rfloor^{p-1}}{\ell S_{\ell} } \mbox{ for } x\in I_\ell.
$$
Since $(p-1)q=p$, we obtain
$$
\E{g^q(X)}=\sum_{\ell=1}^\infty
\sum_{k\in I_\ell} \psi_k \frac{k^{p}}{\ell^q S_\ell^q}
=\sum_{\ell=1}^\infty \frac{1}{\ell^q S_\ell^{q-1}} <\infty.
$$
On the other hand
\[
\E{Xg(X)} \ge \sum_{\ell=1}^\infty
\sum_{k\in I_\ell} \psi_k \frac{k^p}{\ell S_\ell}
=\sum_{\ell=1}^\infty \frac{1}{\ell} =\infty.\qedhere
\]
\end{proof}

We can now proceed with the main proof.

\begin{proof}[Proof of Theorem~\ref{thm:main}, case $\alpha=2$]
We will first show that (i) and (ii) are equivalent. Assume (i) holds, i.e.\
$\E{X^2}<\infty$, and let $T$ satisfy $\E{T}<\infty$. We may assume without
loss of generality that $\E{X}=1$.  By the nonnegativity of the $X_i$, we
have
\begin{equation}\label{eq:converseHR}
 \P(S_T\ge 2n)\le
\P(T\ge n)+ \P(S_n\ge 2n).
\end{equation}
Since $\E{T}<\infty$, the first term on the right is summable in $n$. Since
$\E{X^2}<\infty$ and $\E{X}=1$, Theorem \ref{thm:HRE} with $\epsilon=1$
implies that the second term is also summable.  We conclude that
$\E{S_T}<\infty$.

Now assume (ii).  To show that $X$ has finite second moment, using Lemma
\ref{cor:RealAnalysisGeneral} with $p=q=1$, we need only show that for any
nonnegative function $g$ satisfying $\E{g^2(X)}<\infty$, we have
$\E{Xg(X)}<\infty$. Given such a $g$, consider the integer valued random
variable
\begin{equation}\label{tg}
T_g:=\max\{k\ge 1:g(X_k)\ge k\},
\end{equation}
where $T_g$ is taken to be 0 if the set is empty or $\infty$ is it is
unbounded. We have
$$
\E{T_g}=\sum_{k=1}^\infty  \P(T_g\ge k)\le
\sum_{k=1}^\infty \sum_{\ell=k}^\infty \P(g(X_\ell)\ge \ell)=
\sum_{\ell=1}^\infty \ell\;\P(g(X)\ge \ell).
$$
Since $\E{g^2(X)}<\infty$, the last expression is finite, and hence
$\E{T_g}<\infty$. Thus, by assumption (ii), we have $\E{S_{T_g}}<\infty$. However
\begin{equation}
\begin{aligned}\label{tg-calc}
\E{S_{T_g}}=\E \sum_{k=1}^\infty  X_k\ind[k\le T_g]
&\ge\E{\sum_{k=1}^\infty  X_k \ind[g(X_k)\ge k]}\\
&=\E{\sum_{k=1}^\infty  X\ind[g(X)\ge k]}
\ge
\E{X\lfloor g(X)\rfloor},
\end{aligned}
\end{equation}
so that $\E{X\lfloor g(X)\rfloor}<\infty$, which easily yields
$\E{Xg(X)}<\infty$ as required.

Clearly (ii) implies (iii).  Finally, we proceed as in the proof of
Theorem~\ref{thm:exp} to show (iii) implies (i).  Suppose (i) fails and (iii)
holds.  Taking $T\equiv1$ in (iii) shows that $\E X<\infty$.  Since $\E
X^2=\infty$, Lemma~\ref{cor:RealAnalysisGeneral} implies the existence of a
$g$ with $\E g^2(X)<\infty$ but $\E X g(X)=\infty$.  Let $T_g$ be defined as
in \eqref{tg} above, for this $g$.  The argument above shows that $\E
S_{T_g}=\infty$ while $\E T_g<\infty$, and so the assumption (iii) gives $\E
X_{T_g}<\infty$.  However this contradicts Lemma~\ref{last}.
\end{proof}

We also obtain the following converse of the Hsu-Robbins Theorem due to
Erd\H{o}s.

\begin{corollary}\label{cor:erdos}
Let $X_1,X_2,\ldots$ be i.i.d.\ nonnegative random variables with finite mean
$\mu$.  Write $S_n=\sum_{i=1}^n X_i$ and $X=X_1$. If, for all $\epsilon >0$,
$$
\sum_{n=1}^\infty  \P(|S_n-n\mu|\ge n \epsilon) <\infty,
$$
then $X$ has a finite variance.
\end{corollary}

\begin{proof}
Without loss of generality, we can assume that $\mu=1$. By
Theorem~\ref{thm:main} with $\alpha=2$, it suffices to show that
$\E{S_T}<\infty$ for all $T$ with finite mean. However, this is immediate
from (\ref{eq:converseHR}) -- the first term on the right is summable since
$T$ has finite mean, and the second term is summable by the assumption of the
corollary with $\epsilon=1$.
\end{proof}

\section{The case of $\alpha<2$}

The proof of Theorem \ref{thm:main} in the general case follows very closely
the proof for $\alpha=2$.  We need the follow replacement of
Theorem~\ref{thm:HRE} due to Katz \cite{K}, whose proof we do not give here.
A converse of the results in \cite{K} appears in \cite{BK}.  We will also use
the general case of Lemma~\ref{cor:RealAnalysisGeneral}.

\begin{theorem}[Katz]\label{thm:Katz}
Let $X_1,X_2,\ldots$ be i.i.d.\ random variables satisfying
$\E{|X_1|^t}<\infty$ with $t\ge 1$. If $r>t$, then, for all $\epsilon >0$,
$$
\sum_{n=1}^\infty n^{r-2} \P\bigl(|S_n|\ge n^{r/t} \epsilon\bigr) <\infty.
$$
\end{theorem}

\begin{proof}[Proof of Theorem~\ref{thm:main}, case $\alpha<2$.]
We first prove that (i) implies (ii).  Assume that $\E X^\alpha <\infty$, and
$T$ is an integer valued random variable with $\E T^{1/(\alpha-1)}<\infty$.
Observe that
\begin{equation}\label{eq:converseHRAGAIN}
\P(S_T\ge n)\le
\P\bigl(T\ge \lceil n^{\alpha-1}\rceil\bigr) +
\P\bigl(S_{\lceil n^{\alpha-1}\rceil}\ge n\bigr).
\end{equation}
Since $\P(T\ge \lceil n^{\alpha-1}\rceil)\le \P(T^{1/(\alpha-1)}\ge n)$, the
first term on the right is summable.  For the second, we have
\begin{align*}
\sum_{n=1}^\infty \P\bigl(S_{\lceil n^{\alpha-1}\rceil}\ge n\bigr)&\le
\sum_{k=1}^\infty \sum_{\substack{n\geq 1:\\ \lceil n^{\alpha-1}\rceil=k}}\P(S_{k}\ge n)\\
&\le
\sum_{k=1}^\infty \sum_{\substack{n\geq 1:\\ \lceil n^{\alpha-1}\rceil=k}}
\P\bigl({S_{k}\ge (k-1)^{\frac{1}{\alpha-1}}}\bigr)
\end{align*}
since $\lceil n^{\alpha-1}\rceil=k$ implies that $n\ge (k-1)^{1/(\alpha-1)}$.
It is easy to check that there exists $C_\alpha$ such that for all $k\geq 1$,
$$
\#\{n:\lceil n^{\alpha-1}\rceil=k\}\le C_\alpha k^{\frac{2-\alpha}{\alpha-1}}.
$$
Hence the last double sum is at most
$$
C_\alpha \sum_{k=1}^\infty k^{\frac{2-\alpha}{\alpha-1}}
\P\bigl({S_{k}\ge (k-1)^{\frac{1}{\alpha-1}}}\bigr).
$$
Now using Theorem \ref{thm:Katz} with $t=\alpha$ and $r=\alpha/(\alpha-1)$
and $\epsilon=\frac{1}{2}$ (and noting that $k^{1/(\alpha-1)}/2 \leq
(k-1)^{1/(\alpha-1)}$ for large enough $k$), we conclude that the above
expression is finite.  Hence $\E{S_T}<\infty$, as required.

Next we show that (ii) implies (i).  To show that $X$ has a finite
$\alpha$-moment, using Lemma~\ref{cor:RealAnalysisGeneral}, it suffices to
show that for any nonnegative function $g$ satisfying $\E
g^{\alpha/(\alpha-1)}(X)<\infty$, we have $\E{Xg(X)}<\infty$. Given such a
$g$, consider as before the integer valued random variable
$$
T_g:=\max\{k\ge 1:g(X_k)\ge k\},
$$
where $T_g$ is taken to be 0 if the set in empty or $\infty$ if it is
unbounded. Observe that
\begin{align*}
\sum_{k=1}^\infty k^\frac{2-\alpha}{\alpha-1}  \P(T_g\ge k)&\le
\sum_{k=1}^\infty k^\frac{2-\alpha}{\alpha-1} \sum_{\ell=k}^\infty
\P(g(X_\ell)\ge \ell)\\
&\le \sum_{\ell=1}^\infty \ell^{\frac{1}{\alpha-1}}\P(g(X)\ge \ell) .
\end{align*}
If $\E g^{\alpha/(\alpha-1)}(X)<\infty$ then the last sum is finite and hence
$\E T_g^{1/(\alpha-1)}<\infty$. By assumption (ii) we have
$\E{S_{T_g}}<\infty$.  However, as argued in \eqref{tg-calc}, $\E{S_{T_g}}\ge
\E{X\lfloor g(X)\rfloor}$.  Therefore $\E{X\lfloor g(X)\rfloor}<\infty$, so
$\E{Xg(X)}<\infty$ as required.

Clearly (ii) implies (iii).  Finally, suppose (i) fails and (iii) holds.
Taking $T\equiv1$ in (iii) shows that $\E X<\infty$.  Since $\E
X^\alpha=\infty$, Lemma~\ref{cor:RealAnalysisGeneral} implies the existence
of a $g$ with $\E g^{\alpha/(\alpha-1)}(X)<\infty$ but $\E X g(X)=\infty$.
Then, as before, $T_g$ as defined above gives a contradiction to
Lemma~\ref{last}.
\end{proof}

\section{The dependent case}

\begin{proof}[Proof of Proposition~\ref{thm:GeneralMomentNOTiid}]
Assume (i) holds. If $\E{T^{\alpha/(\alpha-1)}}<\infty$ and $X_1,X_2,\ldots$
are as in (ii), then we can write
$$
S_T\le \sum_{k=1}^T
k^{\frac{1}{\alpha-1}}+
\sum_{k=1}^T
X_k\ind\bigl[X_k\ge  k^{\frac{1}{\alpha-1}}\bigr].
$$
The first sum is at most $T^{\alpha/(\alpha-1)}$ which has finite
expectation. The expectation of the second sum is at most
$$
\E{\sum_{k=1}^\infty X\ind[X\ge  k^{\frac{1}{\alpha-1}}]}
\le \E(X X^{\alpha-1})= \E{X^{\alpha}}<\infty.
$$
Hence $\E{S_T}<\infty$, as claimed in (ii).

Now assume (ii) holds.  To show that $X$ has finite $\alpha$-moment, using
Lemma \ref{cor:RealAnalysisGeneral}, it is enough to show that for any
nonnegative $g$ satisfying $\E{g^{\alpha/(\alpha-1)}(X)}<\infty$, we have
$\E{Xg(X)}<\infty$. It is easily seen that it suffices to only consider $g$
that are integer valued. Given such a $g$, let $T$ be $g(X)$ and let all the
$X_i$ be equal to $X$.  Then $\E T^{\alpha/(\alpha-1)}<\infty$. By (ii),
$\E{S_T}<\infty$. However, by construction $S_T=Xg(X)$, concluding the proof.
\end{proof}

\section*{Acknowledgements}
Most of this work was carried out when the third author was
visiting Microsoft Research at Redmond, WA, and he thanks
the Theory Group for its hospitality. JES also acknowledges the
support of the Swedish Research Council and the Knut and Alice Wallenberg
Foundation.

%\bibliographystyle{plain}
%\bibliography{rw}
%############################################################################################
%############################################################################################
%############################################################################################
%\appendix

\end{document}